\documentclass{article}
\usepackage{lipsum,tocloft}
\usepackage{endnotes}
\let\footnote=\endnote
\usepackage{graphicx}
\usepackage{array,multirow}
\usepackage{amsmath}
\usepackage{amssymb}
\usepackage{amsthm}
\usepackage{comment}
\usepackage{booktabs}
\usepackage{afterpage}

\usepackage{hyperref}
\usepackage{parskip}

\newtheorem{lem}{Lemma}

\newtheorem{theo}{Theorem}

\title{Proof of the integrality of Somos-5}
\author{Anton Enste}
\begin{document}
\maketitle

\textbf{Abstract}\\
In this paper, a simple proof that shows that the Somos-5 sequence
produces only natural numbers is given. We also give simple proof that each
$a_n$ in the Somos-5 sequence is relatively prime to its four
predecessors. The manner in which the integrality is proven is taken
from the 1992 proof for the integrality of the Somos-4 sequence by
Janice. L. Malouf.\footnote{Janice L.\ Malouf. December 1992. An integer sequence from a rational recursion. Volume 110, Issue
  1-3. Seite
  257-261. \\\url{https://www.sciencedirect.com/science/article/pii/0012365X9290714Q?via3Dihub}}
The manner in which it is proven that each $a_n$ is relatively prime
to its four predecessors is taken from a paper from Michael J. Crabb, who proved that
each $a_n$ in the Somos-5 sequence is relatively prime to its two
predecessors.\footnote{Michael J. Crabb, An elementary proof that
    the terms of the standard Somos 5 sequence are integers
    \\\url{https://fulbright.uark.edu/departments/math/_resources/directory/johnduncan-files/M/Somos5-Draft1.pdf}
    Draft1.pdf}
\\\\
\textbf{Definition:}
The Somos-5 sequence named after Micheal Somos, who was first to define it, is given by the following rational recursion
formula with initial values $a_0=a_1=a_2=a_3=a_4=1$:
\begin{equation}\label{def;def}
  a_n=\frac{a_{n-1}a_{n-4}+a_{n-2}a_{n-3}}{a_{n-5}}
\end{equation}
Its first terms, starting with $a_0$ are:
1,1,1,1,1,2,3,5,11,37,83,274,...
\\\\
\begin{lem} Euclid's
  Lemma\footnote{https://mathworld.wolfram.com/EuclidsLemma.html} If
  $d|mn$ and $\text{gcd}(d,m)$ than one finds $d|n$
\end{lem}

\begin{lem} For $a,x,y \in \mathbb{N}$ we have:
  $\text{gcd}(a,x)=\text{gcd}(a,y)=1 \Leftrightarrow
  \text{gcd}(a,xy)=1$ \label{lem;zwei}
\end{lem}
\begin{proof}
  First we prove that:
  \begin{equation*}
    \text{gcd}(a,x)=\text{gcd}(a,y)=1\Rightarrow \text{gcd}(a,xy)=1.
  \end{equation*}
  We assume that there is a number $d$, for which we have $d|a$ and
  $d|xy$.
  Now we have to prove that $d=1$. We want to use Euclid's Lemma, so
  we have to show that $\text{gcd}(d,x)=1$. We can do this with a
  proof by contradiction. So we assume that:
  \begin{equation*}
    \text{gcd} (d,x)=e>1.\\
  \end{equation*}

  We know:
  \begin{equation*}
    \begin{split}
      e|d, e|x \text{ and } e|a \; (\text{because } e|d \text{ and } d|a),\\
      \text{we conclude: }  e|\text{gcd} (a,x)=1,
    \end{split}
  \end{equation*}
  so $e=1$. Now we can use Euclid's Lemma. We have:
  \begin{equation*}
    d|xy \text{ and }\text{gcd} (d,x)=1,
  \end{equation*}
  so we know that $d|y$. We now know the following:
  \begin{equation*}
    \begin{split}
      d|a,\; (\text{because of the definition of $d$})\text{ and }\;d|y,\\
      \text{ so we conclude: }d=1,\text{ because }\text{gcd} (a,y)=1.
    \end{split}
  \end{equation*}
  \\\\
  Now we prove the other direction:
  \begin{equation*}
    \text{gcd}(a,xy)=1 \Rightarrow \text{gcd}(a,x)=\text{gcd}(a,y)=1
  \end{equation*}
  This is easier to prove. If there was a divisor $d>1$, which divided
  $a$ and $x$ or $a$ and $y$, we would have:
  \begin{equation*}
    \text{gcd} (a,xy)\neq 1
  \end{equation*}
\end{proof}

\begin{lem} For $a,b,x,y \in \mathbb{N}$ we have:\\\label{lem;drei}
  $\text{gcd}(a,x)=\text{gcd}(a,y)=\text{gcd}(b,x)=\text{gcd}(b,y)=1
  \Leftrightarrow \text{gcd}(ab,xy)=1$
\end{lem}
\begin{proof} $\Rightarrow$\\
With Lemma \ref{lem;zwei}, we can prove:
\begin{align*}
  &\text{gcd}(a,x)=\text{gcd}(a,y)=1\Rightarrow\text{gcd}(a,xy)=1\\
  \text{and }&\text{gcd}(b,x)=\text{gcd}(b,y)=1\Rightarrow
               \text{gcd}(b,xy)=1
\end{align*}\\\\
Now we substitute $xy=z$:
\begin{equation*}
  \text{gcd}(a,z)=\text{gcd}(b,z)=1
\end{equation*}
Now we can again use Lemma \ref{lem;zwei} and show that
$\text{gcd}(ab,z)=1.$
Because of the substitution $z=xy$, we have $\text{gcd}(ab,xy)=1$.
\\\\
$\Leftarrow$\\
Now it is the same as in Lemma \ref{lem;zwei}: If there was a divisor
which divided $a$ and $x$ or $a$ and $y$ or $b$ and $x$ or $b$ and
$y$, we would have:
\begin{equation*}
  \text{ggT}(ab,xy)\neq 1
\end{equation*}
\end{proof}

\begin{lem} \label{lem;vier} For $x,y \in \mathbb{N}$ we have: $\text{gcd}(x,y)=1
  \Leftrightarrow \text{gcd}(x+y,y)=1$
\end{lem}
\begin{proof}
  First we prove:
  \begin{equation*}
    \text{gcd} (x,y)=1\Rightarrow \text{gcd} (x+y,y)=1
  \end{equation*}
  We assume there is a number $d$ for which we have $d|(x+y)$ and
  $d|y$. Now we only have to prove that $d=1$:
  \begin{align}
    x+y&=d\cdot{m}\label{eq:x+y}\\
    y&=d\cdot{n} \label{eq:y}\\
    \eqref{eq:x+y}-\eqref{eq:y}:\; x+y-y&=d\cdot{m}-d\cdot{n} \nonumber \\
    \Rightarrow \; x&=d\cdot{(m-n)}, \nonumber 
  \end{align}
  with $m,n \in \mathbb{N}$. So we know that $d|x$ and that $d|y$ and
  that $\text{gcd}(x,y)=1$, so $d=1$ and $\text{gcd}(x+y,y)=1.$\\\\

  Now we prove that $\text{gcd}(x,y)=1$ follows from
  $\text{gcd}(x+y,y)$.
  We assume that $d|x$ and $d|y$:
  \begin{align}
    x&=d\cdot{m}\label{eq:xn}\\
    y&=d\cdot{n} \label{eq:yn}\\
    \eqref{eq:xn}+\eqref{eq:yn}:\;x+y&=d\cdot{m}+d\cdot{n}\nonumber\\
    x+y&=d\cdot{(m+n)},\nonumber
  \end{align}
  with $m,n \in \mathbb{N}$. So we conclude that $d|(x+y)$ and $d|y$
  so $d=\text{gcd}(x,y)=1$.
\end{proof}

\begin{theo} For every $a_n$ of the Somos-5 sequence one finds:\\
  \textbf{a)}$\text{gcd}(a_n,a_{n-1})=\text{gcd}(a_n,a_{n-2})=1 \text{
    for every }n \geq 2.$\\
  \textbf{b)}$\text{gcd}(a_n,a_{n-1})=\text{gcd}(a_n,a_{n-2})=\text{gcd}(a_n,a_{n-3})=\text{gcd}(a_n,a_{n-4})=1
  {\text{ for every }n\geq 4.}$
\end{theo}
\begin{proof}
  \textbf{a)} We prove this by induction:\\
  We can easily see that the above is true for $4\geq n \geq 2$ since those
  are all initial values.
  \\\\
  Assume that the above is true for three consecutive $n\geq 2$:
  \begin{equation}\label{eq;first}
    \text{gcd}(a_n,a_{n-1})=\text{gcd}(a_n,a_{n-2})=1
  \end{equation}
  We want to show that
  \begin{equation*}
    \text{gcd}(a_{n+1},a_{n})=\text{gcd}(a_{n+1},a_{n-1})=1
  \end{equation*}
  
  Begin with our assumption, equation \eqref{eq;first}:
  \begin{equation*}
    \text{gcd} (a_{n},a_{n-1})=\text{gcd} (a_{n},a_{n-2})=\text{gcd} (a_{n-1},a_{n-3})=\text{gcd} (a_{n-2},a_{n-3})=1
  \end{equation*}

  We use Lemma \ref{lem;drei} to prove:
  \begin{equation*}
    \text{ggT} (a_{n}a_{n-3},a_{n-1}a_{n-2})=1
  \end{equation*}
  We use Lemma \ref{lem;vier} with $x:=a_{n}a_{n-3}$ and
  $y:=a_{n-1}a_{n-2}$ to prove:
  \begin{align}
     1 &=\text{ggT}(a_{n}a_{n-3}+a_{n-1}a_{n-2}, a_{n-1}a_{n-2})\nonumber\\
      &\stackrel{\text{Def \ref{def;def}}}{=}\text{ggT} (a_{n+1}a_{n-4}, a_{n-1}a_{n-2})\label{theo:a+1,1}\\
       &\stackrel{\text{Lem \ref{lem;drei}}}{=}\text{ggT}(a_{n+1},a_{n-1})\nonumber
  \end{align}
  Now we prove $\text{gcd}(a_{n+1},a_n)=1$. We take from above:
  \begin{equation*}
    \text{ggT} (a_{n}a_{n-3},a_{n-1}a_{n-2})=1
  \end{equation*}
  With Lemma \ref{lem;vier} with $y:=a_{n}a_{n-3}$ and
  $x:=a_{n-1}a_{n-2}$ we get:
  \begin{align}
    1&=\text{ggT}(a_{n}a_{n-3}+a_{n-1}a_{n-2}, a_{n}a_{n-3})\nonumber\\
     &\stackrel{\text{Def \ref{def;def}}}{=}\text{ggT} (a_{n+1}a_{n-4}, a_{n}a_{n-3})\label{theo:a+1,2}\\
    &\stackrel{\text{Lem \ref{lem;drei}}}{=}\text{ggT} (a_{n+1},a_{n})\nonumber
  \end{align}\\
  \textbf{So we have proven that in the Somos-5 sequence, each $a_n$ is
    relatively prime to its two predecessors.}
  \\\\
  The proof for \textbf{b)} is also via induction.
  We can easily see that \textbf{b)} is true for $6 \geq n \geq 4$
  \\\\
  Assume that \textbf{b)} is true for three consecutive $n\geq 4$:
  \begin{equation}\label{equ:asum}
    \text{gcd}(a_n,a_{n-1})=\text{gcd}(a_n,a_{n-2})=\text{gcd}(a_n,a_{n-3})=\text{gcd}(a_n,a_{n-4})=1
  \end{equation}
  We want to show that:
  \begin{equation*}
    \text{gcd}(a_{n+1},a_{n})=\text{gcd}(a_{n+1},a_{n-1})=\text{gcd}(a_{n+1},a_{n-2})=\text{gcd}(a_{n+1},a_{n-3})=1
  \end{equation*}
  The proof for $a_n$ being relatively prime to its first two
  predecessors has already been established in \textbf{a)}, so we only
  have to prove that
  \begin{equation*}
    \text{gcd}(a_n,a_{n-3})=\text{gcd}(a_n,a_{n-4})=1
  \end{equation*}
  Begin with our assumption, equation \eqref{equ:asum}:
  \begin{equation*}
    \text{gcd} (a_{n},a_{n-1})=\text{gcd} (a_{n},a_{n-2})=\text{gcd} (a_{n-1},a_{n-3})=\text{gcd} (a_{n-2},a_{n-3})=1
  \end{equation*}
  We use Lemma \ref{lem;drei} to prove:
  \begin{equation*}
    \text{ggT} (a_{n}a_{n-3},a_{n-1}a_{n-2})=1 \label{eq;die}
  \end{equation*}
  We use Lemma \ref{lem;vier} with $x:=a_{n}a_{n-3}$ and
  $y:=a_{n-1}a_{n-2}$ to prove:
  \begin{align}
     1 &=\text{ggT}(a_{n}a_{n-3}+a_{n-1}a_{n-2}, a_{n-1}a_{n-2})\nonumber\\
      &\stackrel{\text{Def \ref{def;def}}}{=}\text{ggT} (a_{n+1}a_{n-4}, a_{n-1}a_{n-2})\nonumber\\
       &\stackrel{\text{Lem \ref{lem;drei}}}{=}\text{ggT}(a_{n+1},a_{n-2})\nonumber
  \end{align}
  Now we prove $\text{gcd}(a_{n+1},a_{n-3})=1$. We take from equation \eqref{eq;die}:
  \begin{equation*}
    \text{ggT} (a_{n}a_{n-3},a_{n-1}a_{n-2})=1
  \end{equation*}
  With Lemma \ref{lem;vier} with $y:=a_{n}a_{n-3}$ and
  $x:=a_{n-1}a_{n-2}$ we get:
  \begin{align}
    1&=\text{ggT}(a_{n}a_{n-3}+a_{n-1}a_{n-2}, a_{n}a_{n-3})\nonumber\\
     &\stackrel{\text{Def \ref{def;def}}}{=}\text{ggT} (a_{n+1}a_{n-4}, a_{n}a_{n-3})\nonumber\\
    &\stackrel{\text{Lem \ref{lem;drei}}}{=}\text{ggT} (a_{n+1},a_{n-3})\nonumber
  \end{align}\\
  \textbf{So we have proven that in the Somos-5 sequence, each $a_n$ is
    relatively prime to its four predecessors.}
  
\end{proof}

\begin{lem} Given an equation $y=a\pmod{z}$ and another equation \label{lem;gleich}
  $y\cdot b=c\pmod{z}$ and $\text{gcd}(b,z)=1$. We have:
   \begin{equation*}
    \begin{split}
      \text{If } y=0\pmod{z} \Leftrightarrow y\cdot b =0\pmod{z} \\
    \text{ and if } y\neq 0 \pmod{z} \Leftrightarrow y\cdot b \neq
    0\pmod{z}
    \end{split}
  \end{equation*}
\end{lem}
\begin{proof}
  Assume $a\neq 0$ then $y\neq 0\pmod{z}$ and because $b$ and $z$ are
  relatively prime $c\neq 0$, because none of the divisors of $z$ is
  present in $b$.
  \\\\
  If $a=0$ and $y\pmod{z}=0$, then it doesn't matter which natural
  number we multiply $y$ with the solution $c$ will always be 0. This
  simply follows from Euclid's Lemma:
  \begin{equation*}
    z|yb \text{ and } \text{gcd}(z,b)=1 \Rightarrow z|y
  \end{equation*}
\end{proof}
\begin{theo} For every $a_n$ of the Somos-5 sequence, we have
  ${a_{n-1}a_{n-4}+a_{n-2}a_{n-3}=0 \pmod{a_{n-5}}}$ and because of that
  one finds that every $a_n$ is a natural number.
\end{theo}
\begin{proof}
  This proof is also via induction. One also has to know a few
  index shifts of Definition \ref{def;def}:
  \begin{equation*}
    \begin{split}
      a_{n-5}a_{n-10}=a_{n-6}a_{n-9}+a_{n-7}a_{n-8}\\
      a_{n-4}a_{n-9}=a_{n-5}a_{n-8}+a_{n-6}a_{n-7}\\
      a_{n-3}a_{n-8}=a_{n-4}a_{n-7}+a_{n-5}a_{n-6}\\
      a_{n-2}a_{n-7}=a_{n-3}a_{n-6}+a_{n-4}a_{n-5}\\
      a_{n-1}a_{n-6}=a_{n-2}a_{n-5}+a_{n-3}a_{n-4}\\
    \end{split}
  \end{equation*}
It is easy to prove that the first ten elements of the Somos-5
sequence are integers just by looking at them.

Assume that there is an $n>9$ for which all $a_k\in\mathbb{N}$ with
$k \leq n-1$.\\
We want to prove that $a_n\in\mathbb{N}$.
So we want to know if
\begin{equation*}
  a_{n-1}a_{n-4}+a_{n-2}a_{n-3}=0\pmod{a_{n-5}}.
\end{equation*}

According to Lemma \ref{lem;gleich}:
\begin{equation*}
  a_{n-1}a_{n-4}+a_{n-2}a_{n-3}\pmod{a_{n-5}}=a_{n-8}a_{n-9}(a_{n-1}a_{n-4}+a_{n-2}a_{n-3})\pmod{a_{n-5}}
\end{equation*}
because
\begin{equation*}
  \text{gcd}(a_{n-5},a_{n-8})=\text{gcd}(a_{n-5},a_{n-9})=1\stackrel{Lem. \ref{lem;zwei}}{=}\text{gcd}(a_{n-5},a_{n-8}a_{n-9}).
\end{equation*}

So now we have:
\begin{equation}\label{eq:sub}
    \begin{split}
      &a_{n-8}a_{n-9}(a_{n-1}a_{n-4}+a_{n-2}a_{n-3})\\
      &=a_{n-8}a_{n-9}a_{n-1}a_{n-4}+a_{n-8}a_{n-9} a_{n-2}a_{n-3}\\
      &=a_{n-8}a_{n-1}(a_{n-5}a_{n-8}+a_{n-6}a_{n-7})+a_{n-9}
      a_{n-2}(a_{n-4}a_{n-7}+a_{n-5}a_{n-6})\\
      &=a_{n-8}a_{n-1}a_{n-6}a_{n-7}+0+a_{n-9}a_{n-2}a_{n-4}a_{n-7}+0\\
      &=a_{n-8}a_{n-7}(a_{n-2}a_{n-5}+a_{n-3}a_{n-4})+a_{n-9}a_{n-4}(a_{n-3}a_{n-6}+a_{n-4}a_{n-5})\\
      &=a_{n-8}a_{n-7}a_{n-3}a_{n-4}+0+a_{n-9}a_{n-4}a_{n-3}a_{n-6}+0\\
      &=a_{n-3}a_{n-4}(a_{n-8}a_{n-7}+a_{n-9}a_{n-6})\\
      &=a_{n-3}a_{n-4}a_{n-5}a_{n-10}=0 \pmod{ a_{n-5}}\\
  \end{split}
\end{equation}

\textbf{So we have proven that in the Somos-5 sequence, every $a_n\in\mathbb{N}$}
\end{proof}
\pagebreak
\renewcommand{\notesname}{References}

\begingroup
     \parindent 0pt
     \parskip 5ex
     \def\enotesize{\normalsize}
     \theendnotes
\endgroup 
\end{document}